\documentclass[10pt]{amsart}

\usepackage{amsmath}
  \usepackage{paralist}
  \usepackage{graphics} 
  \usepackage{epsfig} 
\usepackage{graphicx}  \usepackage{epstopdf}

 \usepackage[colorlinks=true]{hyperref}
\hypersetup{urlcolor=blue, citecolor=red}

\newtheorem{theorem}{Theorem}[section]
\newtheorem{corollary}[theorem]{Corollary}
\newtheorem{main}{Main Theorem}
\newtheorem*{main*}{Main Theorem}
\newtheorem{lemma}[theorem]{Lemma}
\newtheorem{proposition}[theorem]{Proposition}
\newtheorem{conjecture}[theorem]{Conjecture}

\theoremstyle{definition}
\newtheorem{definition}[theorem]{Definition}
\newtheorem{remark}[theorem]{Remark}

\title[Running heading with forty characters or less]
      {Schmidt Games and Nondense forward Orbits of certain Partially Hyperbolic Systems}

\author[first-name1 last-name1 and first-name2 last-name2]{Weisheng Wu}

\subjclass{}
 \keywords{Schmidt games, Nondense orbits, Hausdorff dimension, Partially Hyperbolic diffeomorphisms, Conformality}

\address{Department of Mathematics, Pennsylvania State University, University Park, PA 16802, USA}
 \email{wu@math.psu.edu}

\begin{document}
\maketitle
\markboth{Schmidt Games and Nondense Forward Orbits of certain Partially Hyperbolic Systems}
{Schmidt Games and Nondense Forward Orbits of certain Partially Hyperbolic Systems}
\renewcommand{\sectionmark}[1]{}

\begin{abstract}
Let $f: M \to M$ be a partially hyperbolic diffeomorphism with conformality on unstable manifolds. Consider a set of points with nondense forward orbit: $E(f, y) := \{ z\in M: y\notin \overline{\{f^k(z), k \in \mathbb{N}\}}\}$ for some $y \in M$. Define $E_{x}(f, y) := E(f, y) \cap W^u(x)$ for any $x\in M$. Following a method of Broderick-Fishman-Kleinbock \cite{BFK}, we show that $E_x(f,y)$ is a winning set of Schmidt games played on $W^u(x)$ which implies that $E_x(f,y)$ has full Hausdorff dimension equal to $\dim W^u(x)$. Furthermore we show that for any nonempty open set $V \subset M$, $E(f, y) \cap V$ has full Hausdorff dimension equal to $\dim M$, by constructing measures supported on $E(f, y)\cap V$ with lower pointwise dimension converging to $\dim M$ and with conditional measures supported on $E_x(f,y)\cap V$. The results can be extended to the set of points with forward orbit staying away from a countable subset of $M$.
\end{abstract}

\section{Introduction}
\subsection{Schmidt games, Diophantine approximation and bounded orbits}
In this paper we shall study the Hausdorff dimension of a set of points with nondense forward orbit under certain partially hyperbolic diffeomorphisms. The main tool that we use is Schmidt games, which were first introduced by W.M.Schmidt in \cite{S} in 1966. A winning set of such games is large in the following sense: it is dense in the metric space, and its intersection with any nonempty open subset has full Hausdorff dimension. Moreover, these properties are stable with respect to countable intersections. See Section $3$ or \cite{S} for more details. Recall that a real number $x$ is called badly approximable if $\|x-p/q\| > c(x)/q^2$ for any rational $p/q$. The classical theorem of Jarnik and Besicovitch states that the set of badly approximable numbers has full Hausdorff dimension $1$. Schmidt proved in \cite{S} that the set of badly approximable numbers is a winning set of Schmidt games and hence has full Hausdorff dimension $1$. Moreover, Schmidt proved the set of badly approximable systems of linear forms is a winning set of Schmidt games and hence has full Hausdorff dimension (cf. \cite{S2}).

There is a well known connection between Diophantine approximation and bounded orbits of flows on Homogeneous spaces. Let $G=SL(2,\mathbb{R})$, $\Gamma=SL(2,\mathbb{Z})$, and $g_t=\left(
                                                \begin{array}{cc}
                                                  e^{-t} & 0 \\
                                                  0 & e^t\\
                                                \end{array}
                                              \right)
$. Then $\alpha$ is badly approximable iff the orbit $\{g_t\left(
                                                           \begin{array}{cc}
                                                             1 & 0 \\
                                                             \alpha & 1 \\
                                                           \end{array}
                                                         \right)
\Gamma \mid t \geq 0\}$ is bounded in $G/\Gamma$. Then Jarnik-Besicovitch theorem can be reformulated as follows: the set of points with bounded orbit under $g_t$ in $G/\Gamma$ has full Hausdorff dimension $3$. Dani applied Schmidt games to generalize the above result to the following two cases:

\begin{theorem} \label{Dani1}
(cf. \cite{D1}) Let $G=SL(n,\mathbb{R})$, $\Gamma=SL(n,\mathbb{Z})$, and
$$g_t= \text{diag}(e^{-t}, \cdots, e^{-t}, e^{\lambda t}, \cdots, e^{\lambda t}),$$
where $\lambda$ is such that the determinant of $g_t$ is $1$. Then the set of points with bounded orbit has full Hausdorff dimension equal to $\dim G$.
\end{theorem}
\begin{theorem}\label{Dani2}
(cf. \cite{D2}) Let $G$ be a connected semisimple Lie group of $\mathbb{R}$-rank $1$, $\Gamma$ a lattice in $G$, and $g_t$ is not quasiunipotent, that is, $\text{Ad\ }g_1$ has an eigenvalue with modulus other than $1$. Then the set of points with bounded orbit has full Hausdorff dimension equal to $\dim G$.
\end{theorem}
Theorem \ref{Dani1} is equivalent to say that the set of  badly approximable systems of linear forms has full Hausdorff dimension. Theorem \ref{Dani2} implies that for a rank one locally symmetric spaces of noncompact type with finite volume, the set of unit vectors tangent to a bounded geodesic has full Hausdorff dimension equal to the dimension of the unit tangent bundle. The result was strengthened by Aravinda and Leuzinger (cf. \cite{AL}) to be that the set of unit vectors on a non-constant $C^1$ curve in the unit tangent sphere at a point for which the corresponding geodesic is bounded, is of Hausdorff dimension $1$ by applying Schmidt games.

Theorem \ref{Dani1} and Theorem \ref{Dani2} motivate a conjecture of Margulis which was solved by Kleinbock and Margulis:
\begin{theorem}\label{Kleinbockmargulis}
(cf. \cite{KM}) Let $G$ be a connected semisimple Lie group without compact factors, $\Gamma$ an irreducible lattice in $G$, $F = \{g_t|t \in \mathbb{R}\}$ a one-parameter nonquasiunipotent subgroup of $G$, and $Y \subset G/\Gamma$ a closed $F$-invariant set of Haar measure $0$. Then for any nonempty open subset $V$ of $G/\Gamma$:
\begin{equation*}
\dim(\{x \in V | Fx \text{\ is bounded and\ } \overline{Fx} \cap Y = \emptyset\}) = \dim(G/\Gamma).
\end{equation*}
\end{theorem}
Recently Kleinbock and Weiss developed a type of modified Schmidt games in \cite{KW1} and proved that the set in Theorem \ref{Kleinbockmargulis} restricted in the unstable horospherical subgroup $H$ is a winning set of modified Schmidt games (cf. \cite {KW2}). It strengthens Theorem \ref{Kleinbockmargulis} since a winning set of modified Schmidt games has stronger properties than having full Hausdorff dimension equal to $\dim H$. On the other hand, Kleinbock and Weiss showed in \cite {KW1} that the set of weighted badly approximable vectors is winning and has full Hausdorff dimension. Recently applying Schmidt games, Dani and Shah \cite{DS} show that a large class of Cantor-like sets of $\mathbb{R}^d$ contains uncountably many badly approximable numbers $(d=1)$ or badly approximable vectors $(d \geq 2)$.

\subsection{Nondense orbits}
It is natural to ask for a dynamical system with hyperbolic behavior on a smooth compact manifold $M$, if it is true that the set of points with orbit staying away from some point has full Hausdorff dimension. Let $f: M \rightarrow M$ be a smooth diffeomorphism, and $y \in M$. Denote
\begin{equation} \label{e:nondenseset}
E(f, y):= \{ z\in M: y\notin \overline{\{f^k(z), k \in \mathbb{N}\}}\}.
\end{equation}
By definition, any point in $E(f, y)$ has a nondense forward orbit in $M$. 

Suppose that $f$ is a partially hyperbolic diffeomorphism, and denote
\begin{equation}\label{e:nondense}
\begin{aligned}
E_{x}(f, y) &:= E(f, y) \cap W^u(x),
\end{aligned}
\end{equation}
for any $x \in M$. We shall prove the following two main theorems in this paper; see Section $2$ for more details and other results.
\begin{main*}
Assume either
\begin{enumerate}
  \item $f$ has one dimensional unstable distribution $E^u$;
OR
  \item if $\dim E^u \geq 2$, $f$ is conformal on unstable manifolds, i.e., for each $x\in M$, the derivative map $T_xf|_{E_x^u}$ is a scalar multiple of an isometry.
\end{enumerate}
Then $E_{x}(f, y)$ is a winning set of Schmidt games played on $W^u(x)$.
\end{main*}

\begin{main*}
For any nonempty open subset $V$ of $M$, $\dim_H(E(f, y)\cap V)=n$, where $n=\dim M$.
\end{main*}

There are some previous results on the Hausdorff dimension of the set of points with nondense orbit. By applying Schmidt games, Dani proved for toral endomorphisms:
\begin{theorem}
(cf. \cite {D}) If $f$ is a semisimple toral endomorphism on $\mathbb{T}^n$, $y\in \mathbb{Q}^n/\mathbb{Z}^n$, then $E(f, y)$ is a winning set of Schmidt games.
\end{theorem}
This result was strengthened by Broderick, Fishman, and Kleinbock:
\begin{theorem}\label{BFK}
(cf. \cite{BFK}) For any toral endomorphism $f$ on $\mathbb{T}^n$, any $y\in M$, $E(f, y)$ is a winning set of Schmidt games.
\end{theorem}
In fact, a more general result than Theorem \ref{BFK} is proved in \cite{BFK}: the lifting of $E(f, y)$ to the universal covering $\mathbb{R}^n$ is a winning set of Schmidt games played on $K$, where $K$ is an absolutely friendly subset of $\mathbb{R}^n$. A main application is to prove that the set of badly approximable systems of affine forms has full Hausdorff dimension; see Corollary $1.4$ in \cite{BFK} for more details.

However, it is hard to apply Schmidt games to the non algebraic dynamical systems. If $f$ is expanding or transitive Anosov, Urba\'{n}ski proved the following theorem by a symbolic approach since there exists the Markov partition for the system:
\begin{theorem}
(cf. \cite{U}) Let $f$ be a $C^2$ expanding endomorphism or a transitive $C^2$ Anosov diffeomorphism, then the set of points with nondense forward orbit or full orbit respectively, has full Hausdorff dimension equal to $\dim M$.
\end{theorem}
We can let $Y$ be a subset of $M$, and $E(f,Y)$ be the set of points with forward orbit staying away from $Y$. Dolgopyat has described in \cite{Dol} that how large the set $Y$ is such that $E(f,Y)$ has full Hausdorff dimension, for piecewise expanding maps of an interval and Anosov diffeomorphisms of two dimensional torus by a symbolic approach.

\subsection{Invariant measures and a conjecture of Katok}
Another motivation of the present paper is a related conjecture of A. Katok. If $f$ is volume preserving and ergodic, it follows that the set $E(f, y)$ has Lebesgue measure zero. Nevertheless, the above results show that the set is still large in the sense of having full Hausdorff dimension. Indeed, for a toral endomorphism or an Anosov diffeomorphism, there exists an abundance of invariant measures, and a Lebesgue measure zero set could be very large under other measures. Katok has conjectured that any smooth diffeomorphism has an abundance of invariant measures with intermediate metric entropies between $0$ and $h_{\text{top}}(f)$, the topological entropy of $f$:

\begin{conjecture}[Katok, cf. \cite {sun3}]
Let $f$ be a $C^r (r>1)$ diffeomorphism on a smooth compact manifold, then for any $\beta \in [0, h_{\text{top}}(f))$, there is an ergodic invariant measure $\mu$ such that $h_{\mu}(f)=\beta$.
\end{conjecture}

Katok proved the conjecture for $C^r(r>1)$ diffeomorphisms on compact surfaces (cf. \cite{Katok1}, \cite{Katok2}). Sun proved the conjecture for certain skew product diffeomorphisms in \cite{sun1}, \cite{sun2}, and a weak form of the conjecture for linear toral automorphisms in \cite{sun3}. A recent result of Quas and Soo on ergodic universality implies the conjecture for linear toral automorphisms, and other topological dynamical systems with almost weak specification, asymptotic entropy expansiveness, and the small boundary property (cf. \cite{quassoo}). The conjecture remains open for partially hyperbolic systems.

In this paper we consider the case when $f$ is a partially hyperbolic diffeomorphism with conformality on unstable manifolds. In Section $5$, to show that $E(f,y)$ has full Hausdorff dimension, we shall construct measures on $M$, with conditional measures on unstable manifolds originated from the process of Schmidt games. Our construction involving Schmidt games might be modified to produce an abundance of invariant measures under $f$, and we hope that these invariant measures can possess at least most of intermediate entropies between $0$ and $h_{\text{top}}(f)$, which is an attempt toward Katok's conjecture in the given special case.

\section{Formulation of results}
Let $M$ be a $n$ dimensional smooth, connected, compact Riemannian manifold without boundary and let $f: M \rightarrow M$ be a $C^{1+\alpha}$ diffeomorphism. $f$ is \emph{partially hyperbolic} (cf. for example \cite{RRU}) if there exists a nontrivial $Tf$-invariant splitting of the tangent bundle $TM= E^s \oplus E^c \oplus E^u$ into so called stable, center, and unstable distributions, such that all unit vectors $v^{\sigma} \in E_x^\sigma$ ($\sigma= c,s,u$) with $x\in M$ satisfy
\begin{equation*}
\|T_xfv^s\| \leq \|T_xfv^c\| \leq \|T_xfv^u\|,
\end{equation*}
and
\begin{equation*}
\|T_xf|_{E^s_x}\| <1, \ \ \ \text{\ and\ \ \ \ } \|T_xf^{-1}|_{E^u_x}\| <1,
\end{equation*}
for some suitable Reimannian metric on $M$.

The distributions $E^s$, $E^c$, $E^u$ are H\"{o}lder continuous over $M$ but in general not Lipschitz continuous. The stable $E^s$ and unstable $E^u$ are integrable: there exist so called stable and unstable foliations $W^s$ and $W^u$ respectively such that $TW^s=E^s$ and $TW^u=E^u$. It is well known that the foliations $W^u$ and $W^s$ are transversally absolutely continuous with bounded Jacobians (cf. \cite{BP}, \cite{PS}, \cite{BW}).

For any $y\in M$, the sets $E(f,y)$ and $E_x(f,y)$ are defined as in \eqref{e:nondenseset} and \eqref{e:nondense}, where $x \in M$ is an arbitrary point and $W^u(x)$ is the (global) unstable manifold through $x$. We shall mainly work on unstable manifolds. The set $E(f,y)$ has been shown to be large in the sense of having full Hausdorff dimension, when $f$ is a toral endomorphism, or a transitive Anosov diffeomorphism(cf. \cite{D}, \cite{BFK}, \cite{U}). This is not surprising for hyperbolic toral automorphisms, since there are many nondense invariant sets, such as the periodic points, and other proper invariant sets studied in \cite{Pr}, \cite{Fr}, \cite{Ma}. 

There are many known results on applying Schmidt games to study the Hausdorff dimension of $E(f,y)$ under algebraic dynamical systems (cf. \cite{D}, \cite{BFK}, \cite{KW2}). In the first part of the present paper, we are able to apply Schmidt games to non-algebraic dynamical systems, i.e., partially hyperbolic diffeomorphisms(P.H.D.) with conformality on the unstable manifolds:

\begin{main}\label{main}
Assume either
\begin{enumerate}
  \item $f$ has one dimensional unstable distribution $E^u$;
OR
  \item if $\dim E^u \geq 2$, $f$ is conformal on unstable manifolds, i.e., for each $x\in M$, the derivative map $T_xf|_{E_x^u}$ is a scalar multiple of an isometry.
\end{enumerate}
Then $E_{x}(f, y)$ is a winning set of Schmidt games played on $W^u(x)$.
\end{main}

\begin{remark}\label{cases}
 The assumption in the above theorem is standing throughout the paper. Clearly, case (1) is contained in case (2). However, it is easier to illustrate the idea of the proof in case (1). So we shall prove for case (1) first, and then sketch the proof for case (2).
\end{remark}

As a consequence of Main Theorem \ref{main} and well known results of Schmidt games, one has:
\begin{corollary}\label{corollary}
$E_{x}(f, y)$ is dense in $W^u(x)$ and for any nonempty open subset $U$ of $W^u(x)$, $\dim_H(E_{x}(f, y)\cap U)=\dim W^u(x)$.
\end{corollary}

Let $Y$ be a countable subset of $M$ and define similarly:
\begin{equation*}
\begin{aligned}
E(f, Y)&:= \{ z\in M: Y \cap \overline{\{f^k(z), k \in \mathbb{N}\}}= \emptyset\}, \\
E_{x}(f, Y) &:= E(f, Y) \cap W^u(x).
\end{aligned}
\end{equation*}
Due to a result of Schmidt \cite{S}, restated in Proposition \ref{intersection}, we have:
\begin{corollary}\label{corollary2}
$E_{x}(f, Y)$ is a winning set of Schmidt games played on $W^u(x)$. Consequently, it is dense in $W^u(x)$ and for any nonempty open subset $U$ of $W^u(x)$, $\dim_H(E_{x}(f, Y)\cap U)=\dim W^u(x)$.
\end{corollary}

However, it is hard to play Schmidt games on the whole manifold $M$. A symbolic dynamics approach was developed by M. Urba\'{n}ski to prove that $E(f, y)$ has full Hausdorff dimension on $M$, when $f$ is an expanding endormorphsim or a transitive Anosov diffeomorphsim (cf. \cite{U}). But if $f$ is a partially hyperbolic diffeomorphism, we don't have Markov partition or a symbolic representation of the system. In the second part of this paper, we develop another approach to prove:

\begin{main}\label{main2}
For any nonempty open subset $V$ of $M$, $\dim_H(E(f, y)\cap V)=n$, where $n=\dim M$.
\end{main}

\begin{corollary} \label{countable}
For any nonempty open subset $V$ of $M$, $\dim_H(E(f, Y)\cap V)=n$.
\end{corollary}

We describe Schmidt games and the properties of a winning set in Section $3$. In Section $4$, we modify a method from \cite{BFK} and prove Main Theorem \ref{main}, with Schmidt games played on an unstable manifold. We also prove Corollary \ref{corollary} and \ref{corollary2} there. In Section $5$, we prove Main Theorem \ref{main2} by constructing measures supported on $E(f,y)$ with lower pointwise dimension converging to $n$. We construct conditional measures supported on $E_x(f,y)$ in a way that has appeared in the proof of a McMullen's result (cf. \cite{Mc}, \cite{U}). Corollary \ref{countable} is proved at the end.

For simplicity, we denote $\sigma=\dim E^\sigma$, where $\sigma=u, c, s$. If $W$ is a submanifold of $M$, then $B^W(x,r)$ denotes the open ball with center $x$ and radius $r$ in $W$, with respect to the induced Riemannian metric on $W$. Similarly $B^u(x,r)$ denotes an open ball in $W^u(x)$. We always denote as $\nu$ the volume measures on various manifolds if it doesn't cause confusion.

\section{Schmidt Games}

Let $(X, d)$ be a complete metric space. We denote as $B(x,r)$ the ball of radius $r$ with center $x$. If $\omega=(x,r) \in X \times \mathbf{R}_{+}$, we also denote $B(\omega):=B(x,r)$.

Schmidt games are played by the two players, Alice and Bob. Fix $0 < \alpha, \beta <1$ and a subset $S \subset X$ (the target set). Bob starts the game by choosing $x_1 \in X$ and $r_1 >0$ hence specifying a pair $\omega_1=(x_1, r_1)$. Then Alice chooses a pair $\omega'_1=(x_1', r_1')$ such that $B(\omega'_1) \subset B(\omega_1)$ and $r_1'=\alpha r_1$. In the next turn, Bob chooses a pair $\omega_2=(x_2, r_2)$ such that $B(\omega_2) \subset B(\omega_1')$ and $r_2= \beta r_1'$, and so on. In the $k$th turn, Bob and Alice choose $\omega_k=(x_k, r_k)$ and $\omega'_k=(x_k', r_k')$ respectively such that
\begin{equation*}
  B(\omega'_k) \subset B(\omega_k) \subset B(\omega_{k-1}'), \ \ r_k= \beta r_{k-1}', \ \ r_k'=\alpha r_k.
\end{equation*}
Thus we have a nested sequence of balls in $X$:
\begin{equation} \label{e:sequence}
B(\omega_1) \supset B(\omega_1') \supset \cdots \supset B(\omega_k) \supset B(\omega_k') \supset \cdots.
\end{equation}
The intersection of all these balls consists of a unique point $x_{\infty} \in X$. We call Alice the winner if $x_{\infty} \in S$, and Bob the winner otherwise. $S$ is called a $(\alpha, \beta)$-winning set if Alice has a strategy to win regardless of how well Bob plays, and we call such a strategy a $(\alpha, \beta; S)$-winning strategy. $S$ is called $\alpha$-winning if it is $(\alpha, \beta)$-winning for any $0 <\beta <1$. $S$ is called a winning set if it is $\alpha$-winning for some $0 < \alpha <1$.

The following nice properties of a winning set are proved in \cite{S}.

\begin{proposition} \label{intersection}
The intersection of countably many $\alpha$-winning sets is $\alpha$-winning.
\end{proposition}

\begin{proposition}
If the game is played on $X=\mathbb{R}^n$ with the Euclidean metric, then any winning set is dense and has full Hausdorff dimension $n$.
\end{proposition}

For a more general metric space $(X, d)$ other than $\mathbb{R}^n$, we suppose $X$ supports a Federer measure:

\begin{definition}
Let $\mu$ be a locally finite Borel measure on a metric space $(X, d)$ and $D > 0$. We call $\mu$ is $D$-Federer if there exists $\rho_0 >0$ such that
$$\mu (B(x, 2\rho)) < D \mu (B(x, \rho)), \ \ \  \forall x \in \text{supp\ } \mu, \ \ \forall 0 < \rho < \rho_0.$$
\end{definition}

Recall that the \emph {lower pointwise dimension} of a measure $\mu$ at $x\in \text{supp\ }\mu$ is defined as:
$$\underline{d}_\mu (x):= \liminf_{\rho \to 0}\frac{\log \mu(B(x, \rho))}{\log \rho},$$
and for an open set $U \subset X$
$$\underline{d}_\mu(U):= \inf_{x\in \text{supp\ }\mu \cap U} \underline{d}_{\mu}(x).$$

The following proposition (Proposition 5.1 in \cite{KW1}) shows that any winning set on supports of a Federer meausre has a positive Hausdorff dimension.

\begin{proposition} \label{HD}
Let $X$ be a complete metric space which is the support of a Federer measure $\mu$. If $S$ is a winning set on $X$, then for any nonempty open set $U \subset X$, one has
$$dim_H(S \cap U) \geq \underline{d}_\mu (U).$$

\end{proposition}

We shall play Schmidt games on unstable manifolds. In Section $5$, we restate Proposition \ref{HD} as Lemma \ref{schmidtgame} with X being a manifold, and $\mu$ being the volume measure on X. We shall prove Lemma \ref{schmidtgame} there, since it is essential to the proof of Theorem \ref{main2}. The proof of Proposition \ref{HD} is essentially the same with small modifications.

\section{Schmidt games on unstable manifolds}
\subsection{Unstable manifold as a playground}

We shall play Schmidt games on $W^u(x)$, where $x$ is an arbitrary point on $M$. Consider the Remannian metric on $W^u(x)$ induced from the one of $M$, and let $d^u$ denote the distance on $W^u(x)$. In this subsection, we study a nice measure supported on $W^u(x)$, which can enable us to estimate the Hausdorff dimension of a winning set on $W^u(x)$. That is, the volume measure $\nu$ on $W^u(x)$.

\begin{definition}\label{power}
We say a measure $\mu$ satisfies a power law, if there exist positive numbers $\delta, c_1, c_2, \rho_0$ such that:
$$c_1\rho^{\delta} \leq \mu (B(z, \rho)) \leq c_2\rho^\delta \ \ \ \ \ \forall z\in \text{supp\ }\mu, \ \ \ \forall 0<\rho<\rho_0.$$
\end{definition}

\begin{lemma}\label{powerlaw}
The volume measure $\nu$ on $W^u(x)$ satisfies a power law, with $\delta=u$.
\end{lemma}

\begin{proof}
Since $\text{supp\ }\nu= W^u(x)$, let us fix an arbitrary $p \in W^u(x)$. Consider the exponential map
$$\exp_p: \tilde{B}(0, 2\rho_0) \to B^u(p, 2\rho_0),$$
where $\tilde{B}(0, 2\rho_0) \subset T_pW^u(x)$, $B^u(p, 2\rho_0) \subset W^u(x)$, and $\rho_0$ is chosen such that $\exp_p$ is a diffeomorphism with $\|T_y\exp_p-Id\| \leq \epsilon_0$ for $\forall y\in \tilde{B}(0, 2\rho_0)$ and some $\epsilon_0 >0$ small enough. Then it is easy to see that for any $z\in B^u(p, \rho_0)$, $0<\rho<\rho_0$,
$$ c_1\rho^u \leq \nu(B^u(z, \rho))\leq c_2\rho^u.$$
for some $c_1, c_2 >0$. We can choose $\rho_0$ to be independent of $p\in W^u(x)$.
\end{proof}

\begin{lemma}\label{Federer}
The volume measure $\nu$ is a Federer measure on $W^u(x)$.
\end{lemma}

\begin{proof}
By Lemma \ref{powerlaw}, there exists $\rho_0 >0$, such that for $\forall 0 < 2\rho < \rho_0$,
$$\nu(B(z, 2\rho)) \leq c_2(2\rho)^u = \frac{c_22^u}{c_1}c_1 \rho^u \leq D\nu(B(z, \rho)),$$
where $D=\frac{c_22^u}{c_1} >0$.
\end{proof}

In \cite{BFK}, the authors define a so called absolutely decaying measure on $\mathbb{R}^n$ so that we can estimate the measure of a neighborhood of a hyperplane in $\mathbb{R}^n$. Here we can loose the condition, and the following lemma is enough for our purpose.

\begin{lemma}
Let $\nu$ be the volume measure on $W^u(x)$. Then there exist $\rho_0 >0$ and some $C>0$ such that
$$\nu(B(x_1, \rho)\cap B(x_2, \epsilon \rho)) < C\epsilon^u \nu(B(x_1, \rho))$$
for any $x_1, x_2 \in W^u(x)$, $ \forall 0 < \rho < \rho_0, \ 0 < \epsilon <1$.
\end{lemma}

\begin{proof}
Let $\rho_0$ be as in Lemma \ref{powerlaw}. Then for $\forall 0 < \rho < \rho_0$,
$$\nu(B(x_1, \rho)\cap B(x_2, \epsilon \rho)) < \nu(B(x_2, \epsilon \rho)) \leq c_2 (\epsilon \rho)^u \leq C \epsilon^u c_1 \rho^u \leq C\epsilon^u \nu(B(x_1, \rho)),$$
where $C=\frac{c_2}{c_1}>0$.
\end{proof}

\begin{remark}
In next subsection, we consider a partially hyperbolic system with one dimensional unstable manifolds, i.e., $u=1$. Since the volume measure is the length, we have a more specific choice for the constants in the previous lemmas: $c_1=c_2=1, D=2, C=1$.
\end{remark}

The next lemma is similar to Lemma $3.4$ in \cite{BFK} which is crucial to our proof of Main Theorem \ref{main}. It guarantees that while playing Schmidt games on $W^u(x)$, at each turn Alice can stay away from at least some of the specified balls.

\begin{lemma} \label{Ne}
Let $C, D, \rho_0$ be as in the previous lemmas, and
\begin{equation} \label{e:winning}
\begin{aligned}
0< \alpha < \frac{1}{2}(\frac{1}{CD})^{\frac{1}{u}}.
\end{aligned}
\end{equation}
There exists $\epsilon = \epsilon(C, D) \in (0, 1)$, such that if $x_1 \in W^u(x)$, $0 <\rho < \rho_0$, $y_1, y_2, \cdots, y_N$ are $N$ points in $W^u(x)$, there exists $x_2 \in W^u(x)$ such that
$$B(x_2, \alpha \rho) \subset B(x_1, \rho),$$
and
$$B(x_2, \alpha \rho)\cap B(y_i, \alpha \rho)= \emptyset$$
for at lest $\lceil \epsilon N \rceil$ (the smallest integer $\geq \epsilon N$) of the points $y_i, \ 1\leq i \leq N$.
\end{lemma}

\begin{proof}
Let $A_i:= B(x_1, (1-\alpha) \rho) \setminus B(y_i, 2\alpha \rho)$, $1 \leq i \leq N$. Then
\begin{equation}\label{e:estimate}
\begin{aligned}
\frac{\nu(A_i)}{\nu(B(x_1, \rho))} &\geq \frac{\nu(B(x_1, (1-\alpha) \rho))}{\nu(B(x_1, \rho))}-\frac{\nu(B(y_i, 2\alpha \rho)}{\nu(B(x_1, \rho))}
\\& > \frac{1}{D}-C(2\alpha)^u
\\&:= \epsilon >0.
\end{aligned}
\end{equation}
Thus,
$$\int_{B(x_1, \rho)}\sum_{i=1}^N \chi_{A_i}(x)d\nu(x) \geq N\epsilon \nu (B(x_1, \rho)).$$
Hence by the mean value theorem, there exists $x_2$ with $\sum_{i=1}^N \chi_{A_i}(x_2)\geq N\epsilon$, i.e., there exist $j_1, \cdots, j_k$, where $k \geq \lceil \epsilon N\rceil$, such that $x_2 \in \cap_{i=1}^k A_{j_i}$ which implies $B(x_2, \alpha \rho) \subset B(x_1, \rho)$ and $B(x_2, \alpha \rho)\cap B(y_{j_i}, \alpha \rho)= \emptyset$.
\end{proof}

\begin{remark}
If we choose $C=\frac{c_2}{c_1}$, $D=\frac{c_22^u}{c_1}$, then $0< \alpha < \frac{1}{4}(\frac{c_1}{c_2})^{\frac{2}{u}}$. In the case when $\dim W^u(x_0)=1$, we choose $c_1=c_2=1$, so $0 < \alpha < \frac{1}{4}$.
\end{remark}
\begin{remark}
The estimation in the above proof is a slight modification of the proof of Lemma $3.4$ in \cite{BFK} which is for an absolutely friendly measure on $\mathbb{R}^n$. In fact we can do a better estimation than \eqref{e:estimate} by using Lemma \ref{powerlaw}. That is,
\begin{equation*}
\begin{aligned}
\frac{\nu(A_i)}{\nu(B(x_1, \rho))} &\geq \frac{\nu(B(x_1, (1-\alpha) \rho))}{\nu(B(x_1, \rho))}-\frac{\nu(B(y_i, 2\alpha \rho))}{\nu(B(x_1, \rho))}
\\& \geq \frac{c_1(1-\alpha)^u}{c_2}-\frac{c_2(2\alpha)^u}{c_1}
\\& =\frac{c_1}{c_2}[(1-\alpha)^u-(2\alpha)^u (\frac{c_2}{c_1})^2].
\end{aligned}
\end{equation*}
So we can pick any $\alpha$ with $(1-\alpha)^u > (2\alpha)^u (\frac{c_2}{c_1})^2$, i.e., $0<\alpha< \frac{1}{1+2(\frac{c_2}{c_1})^{\frac{2}{u}}} <\frac{1}{3}$. If in Lemma \ref{powerlaw} we choose $\rho_0$ very small, then $c_1$ and $c_2$ are very close to each other. Hence we can pick arbitrary $0< \alpha <\frac{1}{3}$.

\end{remark}

We will restate and prove Main Theorem \ref{main} according to the two different cases in Remark \ref{cases} separately in the next two subsections. The strategy to prove Main Theorem \ref{main} is similar as the one of Theorem $4.1$ in \cite{BFK}. The difficulty here is that $f|_{W^u(x)}$ is nonlinear and hence has different expanding rates in different directions and at different points. Let $W$ be a local manifold passing through $y$ transversally to the foliation $W^u$ ($\dim W=n-u$). We call an open $c$-rectangle ($c$ is very small) at $y$ the set
\begin{equation*}
\Pi(c) := \Pi(y, W, c):= \bigcup_{z\in B^W(y,c/2)} B^u(z, c/2)
\end{equation*}
Denote $I_k=I_k(c)$ a connected component of $f^{-k}(\Pi(c)) \cap W^u(x)$ on $W^u(x)$, $k \geq 0$. Note that there may exist more than one connected components for a same $k$.
\subsection{P.H.D. with 1-dim unstable manifolds}

In this subsection we consider a partially hyperbolic diffeomorphism $f$ with one dimensional unstable manifolds. Since $M$ is compact, and $E^u$ is a H\"{o}lder continuous distribution on $M$, we can suppose
\begin{equation*}
\sigma_1 \leq \|T_zf|_{E_z^u}\| \leq \sigma_2,   \ \ \text{for any \ }z \in M.
\end{equation*}
We denote $\|I_k\|$ the length of $I_k$ on the unstable manifold. For simplicity we also denote $f$ for $f|_{W^u(x)}$, and $f'(z)$ for $\|T_zf|_{E_z^u}\|$. A useful tool is the following bounded distortion property.

\begin{lemma}\label{BD}(Bounded Distortion)
For any $z_1, z_2 \in I_k(c)$, one has
$$\frac{1}{K} \leq \frac{(f^k)'(z_1)}{(f^k)'(z_2)} \leq K,$$
for some $K=K(c)$, and $K \to 1$ as $c \to 0$.
\end{lemma}

\begin{proof}
Since $E^u$ is H\"{o}lder continuous and so is $\log f'$, there exist $l>0$, $0 < \theta <1$ such that $\|\log f'(z_1)-\log f'(z_2)\| \leq l (d^u(z_1, z_2))^\theta$ for nearby $z_1$ and $z_2$. Recall that $\sigma_1 < f'(z) < \sigma_2$ for any $z \in W^u(x_0)$. For any $z_1, z_2 \in I_k$, since $d^u(f^k(z_1), f^k(z_2))\leq c$, one has
\begin{equation}\label{e:distance}
d^u(f^i(z_1), f^i(z_2)) \leq \frac{c}{(\sigma_1)^{k-i}}, \ \ \text{for\ \ } \forall 0 \leq i \leq k.
\end{equation}
Thus,
\begin{equation} \label{e:BD}
\begin{aligned}
\|\log \frac{(f^k)'(z_1)}{(f^k)'(z_2)}\| &\leq \sum_{i=0}^{k-1}\|\log f'(f^i(z_1)) - \log f'(f^i(z_2))\|
 \\ &\leq \sum_{i=0}^{k-1}l(d^u(f^i(z_1), f^i(z_2)))^\theta
\\&\leq \sum_{i=0}^{k-1}\frac{lc^\theta}{(\sigma_1)^{\theta (k-i)}} \leq \frac{lc^\theta}{(\sigma_1)^\theta-1}.
\end{aligned}
\end{equation}
Hence $$\frac{1}{K} \leq \frac{(f^k)'(z_1)}{(f^k)'(z_2)}\leq K,$$
where $K=\exp(\frac{lc^\theta}{(\sigma_1)^\theta-1})$.
\end{proof}

\begin{theorem}\label{1dim}
Suppose $\dim E^u=1$. Let $\alpha$ be as in \eqref{e:winning}. Then $E_x(f, y)$ is $\alpha$-winning on $W^u(x)$.
\end{theorem}

\begin{proof}[Proof of Theorem \ref{1dim}]
Pick an arbitrary $0 < \beta<1$. Let $\epsilon$ be as in \eqref{e:estimate}. Choose $r \in \mathbf{N}$ large enough, such that
\begin{equation} \label{e:ner}
(1-\epsilon)^rN <1, \text{\ where\ } N=\lfloor \frac{\log K+r\log(\frac{1}{\alpha\beta})}{\log\sigma_1}\rfloor+3.
\end{equation}

Fix $L>0$. Regardless of the initial move of Bob, Alice can make arbitrary moves waiting until Bob chooses a ball of radius $\rho =\min\{\rho_0, \frac{L}{100}\}$. Hence without loss of generality, we may assume $B(\omega_1)$ has radius

\begin{equation} \label{e:firstmove}
\rho=\min\{\rho_0, \frac{L}{100}\}.
\end{equation}

Choose $c'>0$ small enough, such that:

\begin{enumerate}
  \item $1 < K=K(c') \leq 1+\eta$ where $\eta >0$ is very small,
  \item For any $z\in \Pi(c')$, $W^u_L(z) \cap \Pi(c')$ has only one connected component, which contains $z$.
\end{enumerate}
Now choose $0 <c\ll c'$ such that:

\begin{equation}\label{e:c}
c\leq \frac{\alpha c'(\alpha\beta)^{2r-1}}{100K},
\end{equation}
and
\begin{equation}\label{e:j=0}
c<\alpha \rho(\alpha\beta)^{2r-1}.
\end{equation}
Note that the choice of $c$ depends heavily on $\rho$, i.e., the initial move of Bob.

Now we describe a strategy for Alice to win the $(\alpha, \beta)$-Schmidt games on $W^u(x_0)$ with target set $S=E_x(f, y)$. We claim that for each $j\in \mathbb{N}$, Alice can ensure for any $x\in B(\omega'_{r(j+1)})$, and any $I_k=I_k(c)$ with $\|I_k\| \geq \alpha \rho(\alpha\beta)^{(j+2)r-1}$, she has $x \notin I_k$. This will imply $\cap_iB(\omega'_i) \subset (\cup_{k} I_k)^c \subset E_x(f, y)$, and finish the proof.

We prove the claim by induction on $j$. At $j=0$ step, by \eqref{e:j=0} one has for any $k \in\mathbb{N}$
$$\|I_k\| \leq c < \alpha \rho(\alpha\beta)^{2r-1}.$$
So there is no $I_k$ for Alice to avoid and she can play arbitrarily at the first $r$ turns.

Assume the claim is true for $0, 1, \cdots, j-1$. Now we consider the $j$th step. Suppose Bob already picked $B(\omega_{jr+1})$. In this step (containing $r$ turns of play), Alice only need to avoid the $I_k$'s satisfying

\begin{equation} \label{e: mainequality}
\alpha \rho(\alpha\beta)^{(j+2)r-1} \leq \|I_k\| < \alpha \rho(\alpha\beta)^{(j+1)r-1},
\end{equation}
and
\begin{equation} \label{e: intersection}
I_k \cap B(\omega_{jr+1})  \neq \emptyset.
\end{equation}

In the following two lemmas we consider the $I_k$'s satisfying
\begin{equation} \label{e: subset}
I_k \subset B(\omega_{jr+1}).
\end{equation}
We will prove there are at most $(N-2)$ $I_k$'s satisfying \eqref{e: mainequality} and \eqref{e: subset}. Then considering the intersection at two endpoints of $B(\omega_{jr+1})$, there are at most $N$ $I_k$'s satisfying \eqref{e: mainequality} and \eqref{e: intersection}.

\begin{lemma}\label{onek}
For each $k$, there exists at most one $I_k$ satisfying both \eqref{e: mainequality} and \eqref{e: subset}.
\end{lemma}

\begin{proof}[Proof of Lemma \ref{onek}]
Assume $I_k, I'_k \subset B(\omega_{jr+1})$ are two different intervals satisfying both \eqref{e: mainequality} and \eqref{e: subset}. Then there exists an interval $J_k$ such that $I_k \subset J_k \subset B(\omega_{jr+1})$, $f^k(J_k) \subset \Pi(c')$ and
\begin{equation*}
\|f^k(J_k)\| \geq c'/2.
\end{equation*}
Reall that $$\|I_k\| \geq \alpha \rho(\alpha\beta)^{(j+2)r-1}.$$
By the bounded distortion Lemma \ref{BD}, and the Mean Value theorem, one has
\begin{equation*}
\frac{\|J_k\|}{\|I_k\|} \geq \frac{\|f^k(J_k)\|}{K\|f^k(I_k)\|} \geq \frac{c'}{2Kc}.
\end{equation*}
Thus by \eqref{e: mainequality}, and \eqref{e:c},
\begin{equation}\label{e: B}
\begin{aligned}
\|J_k\| \geq \frac{c'}{2Kc}\|I_k\| &\geq \frac{c'}{2Kc}\alpha \rho(\alpha\beta)^{(j+2)r-1}\\
 &\geq \frac{50}{\alpha(\alpha\beta)^{2r-1}}\cdot \alpha \rho(\alpha\beta)^{(j+2)r-1}=50 \rho (\alpha\beta)^{jr}.
\end{aligned}
\end{equation}
But $\|J_k\| \leq \|B(\omega_{jr+1})\|= \rho (\alpha\beta)^{jr}$, a contradiction to \eqref{e: B}. This finishes the proof of the lemma.
\end{proof}

\begin{lemma} \label{manyk}
There are at most $(N-2)$ $k$'s satisfying both \eqref{e: mainequality} and \eqref{e: subset}.
\end{lemma}

\begin{proof}[Proof of Lemma \ref{manyk}]
Let $k_1$ and $k_2$ be the minimal and the maximal ones respectively among all $k$'s satisfying both \eqref{e: mainequality} and \eqref{e: subset}. Then $k_1 \leq k_2$ and
\begin{equation} \label{e: k1k2}
\|f^{k_1}(I_{k_1})\| = \|f^{k_2}(I_{k_2})\|=c.
\end{equation}
The argument in the proof of Lemma \ref{onek} in fact implies:

\begin{equation*}
f^{k_1}(I_{k_2})\subset f^{k_1}(B(\omega_{jr+1})) \subset \Pi(c').
\end{equation*}
Thus,
\begin{equation} \label{e: k1/k2}
\|f^{k_1}(I_{k_2})\| \leq \frac{\|f^{k_2}(I_{k_2})\|}{(\sigma_1)^{k_2-k_1}}= \frac{c}{(\sigma_1)^{k_2-k_1}},
\end{equation}
and by the bounded distortion Lemma \ref{BD} and the Mean Value theorem,

\begin{equation} \label{e: k2//k1}
\frac{\|I_{k_2}\|}{\|I_{k_1}\|} \leq \frac{K\|f^{k_1}(I_{k_2})\|}{\|f^{k_1}(I_{k_1})\|} \leq \frac{K}{(\sigma_1)^{k_2-k_1}}
\end{equation}
where the last inequality follows from \eqref{e: k1k2} and \eqref{e: k1/k2}. Combining \eqref{e: mainequality} and \eqref{e: k2//k1}, one has

\begin{equation*}
\begin{aligned}
\alpha \rho(\alpha\beta)^{(j+2)r-1} &\leq \|I_{k_2}\|  \\
&\leq \frac{K}{(\sigma_1)^{k_2-k_1}}\|I_{k_1}\| \\
&\leq  \frac{K}{(\sigma_1)^{k_2-k_1}}\alpha \rho(\alpha\beta)^{(j+1)r-1},
\end{aligned}
\end{equation*}
which implies:
\begin{equation*}
(\sigma_1)^{k_2-k_1} \leq \frac{K}{(\alpha\beta)^r}.
\end{equation*}
Hence
\begin{equation*}
k_2-k_1 \leq \lfloor \frac{\log K+r\log(\frac{1}{\alpha\beta})}{\log\sigma_1}\rfloor=N-3
\end{equation*}
which finishes the proof of the lemma.
\end{proof}

Now Alice can apply Lemma \ref{Ne} $r$ times to choose $B(\omega'_{jr+1}), \cdots, B(\omega'_{(j+1)r})$ respectively, to avoid all the $I_k$'s satisfying both \eqref{e: mainequality} and \eqref{e: intersection}. Indeed, by \eqref{e: mainequality},
$$I_k \subset B(y_i, \alpha\rho(\alpha\beta)^{(j+1)r-1})$$
for some $y_i, 1 \leq i \leq N$.
Since $N(1-\epsilon)^r <1$ by \eqref{e:ner}, one can have

$$B(\omega'_{(j+1)r}) \cap I_k =\emptyset$$
for all $I_k$'s satisfying both \eqref{e: mainequality} and \eqref{e: intersection}, which implies the claim.
\end{proof}

\subsection{P.H.D. with conformality on unstable manifolds}
We generalize above result to the P.H.D. with higher dimensional unstable manifolds, and with $Tf|_{E^u}$ conformal, i.e.,
for each $z\in M$, the derivative map $T_zf|_{E_z^u}$ is a scalar multiple of an isometry. We denote the magnitude of this scalar by $\|T_z^uf\|$ for simplicity, then obviously $\|T_z^uf\| >1$. Suppose $\sigma_1 \leq \|T_z^uf\| \leq \sigma_2$ for any $z\in M$. Then we also have:

\begin{lemma}\label{BDconformal}(Bounded Distortion)
For any $z_1, z_2 \in I_k(c)$, one has
$$\frac{1}{K} \leq \frac{\|T_{z_1}^uf^k\|}{\|T_{z_2}^uf^k\|} \leq K,$$
for some $K=K(c)$, and $K \to 1$ as $c \to 0$.
\end{lemma}
\begin{proof}
Note that $z \mapsto \log \|T^u_zf\|$ is H\"{o}lder continuous on $M$, and \eqref{e:distance} is still true. Hence the same estimation in \eqref{e:BD} is valid which implies the desired statement.
\end{proof}

For preimages of a small ball on unstable manifolds, the ratio of major radius $R$ and minor radius $r$ are bounded. For instance, say $I_k= f^{-k}(B^u(z,c/2))$ for some $z \in \Pi(c)$. Denote
\begin{equation*}
R(I_k):= \max_{w \in \partial I_k}d^u(f^{-k}(z), w)
\end{equation*}
and
\begin{equation*}
r(I_k):= \min_{w \in \partial I_k}d^u(f^{-k}(z), w),
\end{equation*}
then:
\begin{lemma}
$$\frac{R(I_k)}{r(I_k)} \leq K.$$
\end{lemma}

\begin{proof}
By the Mean Value theorem and the bounded distortion Lemma \ref{BDconformal},
\begin{equation*}
 \frac{d^u(f^{-k}(z), w_1)}{d^u(f^{-k}(z), w_2)}\leq \frac{\|T_{w_1'}^uf^k\|^{-1}d^u(z, f^k(w_1))}{\|T_{w_2'}^uf^k\|^{-1}d^u(z, f^k(w_2))}\leq K\frac{c/2}{c/2}=K,
\end{equation*}
for any $w_1, w_2 \in \partial I_k$ and some $w_1', w_2' \in I_k$. Hence $\frac{R(I_k)}{r(I_k)} \leq K$.
\end{proof}

\begin{theorem}\label{conformal}
Suppose $\dim E^u \geq 2$ and $Tf|_{E^u}$ is conformal. Let $\alpha$ be as in \eqref{e:winning}. Then $E_x(f, y)$ is $\alpha$-winning on $W^u(x)$.
\end{theorem}

\begin{proof}[Proof of Theorem \ref{conformal}]
We follow the scheme in the proof of Theorem \ref{1dim}. Instead of $\|I_k\|$ (the diameter of $I_k$), one has to estimate $R(I_k)$ and $r(I_k)$. It turns out we don't need to modify anything about the choice of $r$, $N$, $\rho$, $c'$ and $c$. They remain the same as in Theorem \ref{1dim}.

Similarly, we claim that for each $j\in \mathbb{N}$, Alice can ensure for any $x\in B(\omega'_{r(j+1)})$, and any $I_k=I_k(c)$ with $R(I_k) \geq \alpha \rho(\alpha\beta)^{(j+2)r-1}$, she has $x \notin I_k$. The $j=0$ step follows from \eqref{e:j=0} as well. In $j$th step, Alice need to avoid the $I_k$'s satisfying

\begin{equation} \label{e: mainequalitycon}
\alpha \rho(\alpha\beta)^{(j+2)r-1} \leq R(I_k) < \alpha \rho(\alpha\beta)^{(j+1)r-1},
\end{equation}
and \eqref{e: intersection}. Similarly we have the following two lemmas which imply that there are at most $N$ such $I_k$'s. Here we consider
\begin{equation}\label{e: subsetcon}
I_k \subset B(\tilde{\omega}_{jr+1})
\end{equation}
where $B(\tilde{\omega}_{jr+1})$ denotes the concentric ball with $B(\omega_{jr+1})))$ of radius $\rho(\alpha\beta)^{jr}(1+2\alpha(\alpha\beta)^{r-1}).$ Hence any $I_k$ satisfying both \eqref{e: mainequalitycon} and \eqref{e: intersection} also satisfies \eqref{e: subsetcon}.

\begin{lemma}\label{onekcon}
For each $k$, there exists at most one $I_k$ satisfying both \eqref{e: mainequalitycon} and \eqref{e: subsetcon}.
\end{lemma}

\begin{proof}[Proof of Lemma \ref{onekcon}]
Assume there are two different $I_k, I'_k \subset B(\tilde{\omega}_{jr+1})$ satisfying both \eqref{e: mainequalitycon} and \eqref{e: subsetcon}. Then there exists a ball $J_k$ (concentric with $I_k$) such that $I_k \subset J_k \subset B(\tilde{\omega}_{jr+1})$, $f^k(J_k) \subset \Pi(c')$ and
\begin{equation*}
R(f^k(J_k)) \geq c'/4.
\end{equation*}
By bounded distortion Lemma \ref{BDconformal}, and the Mean Value theorem:

\begin{equation*}
\begin{aligned}
\frac{R(J_k)}{R(I_k)} \geq \frac{R(f^k(J_k))}{KR(f^k(I_k))} \geq \frac{c'}{4Kc}.
\end{aligned}
\end{equation*}
Thus by \eqref{e: mainequalitycon}, and \eqref{e:c},
\begin{equation}\label{e: Bcon}
\begin{aligned}
R(J_k) \geq R(I_k)\frac{c'}{4Kc} &\geq \alpha\rho(\alpha\beta)^{(j+2)r-1}\cdot \frac{c'}{4Kc} \\
 & \geq 25\rho (\alpha\beta)^{jr}.
\end{aligned}
\end{equation}
But $R(J_k) \leq R(B(\tilde{\omega}_{jr+1}))= \rho(\alpha\beta)^{jr}(1+2\alpha(\alpha\beta)^{r-1})$, a contradiction to \eqref{e: Bcon}.

\end{proof}

\begin{lemma} \label{manykcon}
There are at most $N$ $k$'s satisfying both \eqref{e: mainequalitycon} and \eqref{e: subsetcon}.
\end{lemma}

\begin{proof}[Proof of Lemma \ref{manykcon}]
The argument in the proof of Lemma \ref{onekcon} in fact implies:

\begin{equation*}
f^{k_1}(I_{k_2})\subset f^{k_1}(B(\tilde{\omega}_{jr+1})) \subset \Pi(c').
\end{equation*}
Hence by bounded distortion Lemma \ref{BDconformal}, and the Mean Value theorem,
\begin{equation}\label{e: k2/k1}
\frac{R(I_{k_2})}{R(I_{k_1})} \leq \frac{KR(f^{k_1}(I_{k_2}))}{r(f^{k_1}(I_{k_1}))} \leq \frac{K}{r(f^{k_1}(I_{k_1}))}\cdot \frac{R(f^{k_2}(I_{k_2}))}{(\sigma_1)^{k_2-k_1}} = \frac{K}{(\sigma_1)^{k_2-k_1}}.
\end{equation}
The last equality follows from $R(f^{k_2}(I_{k_2}))=r(f^{k_1}(I_{k_1}))=c/2$. Hence:

\begin{equation*}
\begin{aligned}
\alpha \rho(\alpha\beta)^{(j+2)r-1} &\leq R(I_{k_2})  \\
&\leq \frac{K}{(\sigma_1)^{k_2-k_1}}R(I_{k_1}) \\
&\leq  \frac{K}{(\sigma_1)^{k_2-k_1}}\alpha \rho(\alpha\beta)^{(j+1)r-1},
\end{aligned}
\end{equation*}
which implies:
\begin{equation*}
(\sigma_1)^{k_2-k_1} \leq \frac{K}{(\alpha\beta)^r}.
\end{equation*}
Hence:
\begin{equation*}
k_2-k_1 \leq \lfloor \frac{\log K+r\log(\frac{1}{\alpha\beta})}{\log\sigma_1}\rfloor=N-3
\end{equation*}
which finishes the proof of the lemma.
\end{proof}

The same argument of applying Lemma \ref{Ne} finishes the proof of the theorem.
\end{proof}

\begin{proof}[Proof of Corollary \ref{corollary}]
The volume measure $\nu$ is a Federer measure on $W^u(x)$ by Lemma \ref{Federer}. By Lemma \ref{powerlaw}:
\begin{equation*}
\underline{d}_\nu (z):= \liminf_{\rho \to 0}\frac{\log \nu(B^u(z, \rho))}{\log \rho} = u.
\end{equation*}
Hence $\underline{d}_\nu(U)=u$. So by Proposition \ref{HD}, one has $\dim_H(E_{x}(f, y)\cap U)=u$. Obviously as a winning set $E_{x}(f, y)$ is dense in $W^u(x)$.
\end{proof}

\begin{proof}[Proof of Corollary \ref{corollary2}]
By Proposition \ref{intersection}, $E_x(f, Y)$ is a winning set. Then the proof of Corollary \ref{corollary} also applies to $E_x(f, Y)$.
\end{proof}

\section{Full Hausdorff dimension on the manifold}

Our aim in this section to to derive that $\dim_H E(f,y)=n$, where $n=\dim M$. Recall that for any $x\in M$, $\dim_H(E_x(f, y))=u$. We can't apply directly the Marstrand Theorem to get the estimate $\dim_H(A \times B) \geq \dim_H A+ \dim_H B$, as the unstable foliation is only absolutely continuous, but not Lipschitz continuous. Instead we will apply the following easy half of Frostman's Lemma (cf. \cite{F}):

\begin{lemma} \label{frostman}
Let $F$ be a Borel subset of a Riemannian manifold $X$, let $\mu$ be a Borel measure on $F$, and let $0 <c< \infty$ be a constant. If
\begin{equation} \label{frostmancondition}
\limsup_{r \to 0} \frac{\mu(B(x, r))}{r^h} < c \ \text{\ for all\ } x\in F,
\end{equation}
then $\dim_H F \geq h$.
\end{lemma}

In the remaining of this section, we shall construct a Borel measure $\mu$ on $E(f,y)$ satisfying the condition in \eqref{frostmancondition}. Fix $x_0 \in M$. Suppose $W$ is a local smooth foliation transversal to $W^u$ near $x_0$. Recall that a $\delta$-rectangle at $x_0$ is defined as
\begin{equation} \label{e:set}
\Pi(x_0) := \Pi(x_0, W, \delta):= \bigcup_{x\in \overline{B^W(x_0,\delta/2)}} B^u(x, \delta/2)
\end{equation}
The measure $\mu$ will be constructed to be supported on $\Pi(x_0)\cap E(f,y)$. At the first stage we construct the "conditional" measures on local unstable manifolds.

\subsection{Construction of measure family $\{\mu_x\}$}

The following construction for producing fractal sets has been described in \cite{KM}. Let $A_0 \subset X$ be a compact subset of a Riemannian manifold $X$ and let $m$ be a Borel measure on $X$. For any $l\in \mathbb{N}_0$, let $\mathcal{A}_l$ denote a finite collection of compact subsets of $A_0$ satisfying the following conditions:
\begin{equation} \label{e:tree1}
\mathcal{A}_0 =\{A_0\} \text{\ and\ } m(A_0) >0;
\end{equation}
\begin{equation} \label{e:tree2}
\text{For\ } \forall l \in \mathbb{N}, \ \text{if\ } A, B \in \mathcal{A}_l \text{\ and\ } A \neq B, \text{\ then\ } m(A \cap B)=0;
\end{equation}
\begin{equation} \label{e:tree3}
\text{For\ } \forall l \in \mathbb{N}, \text{\ every element }B \in \mathcal{A}_l \text{\ is contained in an element\ } A \in \mathcal{A}_{l-1}.
\end{equation}
Let $\mathcal{A}$ be the union of subcollections $\mathcal{A}_l, l\in \mathbb{N}_0$. Then $\mathcal{A}$ is called \emph{tree-like} if it satisfies conditions \eqref{e:tree1},\eqref{e:tree2},\eqref{e:tree3}. $\mathcal{A}$ is called \emph{strongly tree-like} if it is tree-like and in addition:
\begin{equation} \label{e:tree4}
d_l(\mathcal{A}):= \sup\{\text{diam}(A): A \in \mathcal{A}_l\} \to 0, \text{\ as\ } l \to \infty.
\end{equation}

For each $l \in \mathbb{N}_0$, denote $\mathbf{A}_l=\cup_{A\in \mathcal{A}_l}A$. Then we can define the limit set of $\mathcal{A}$ to be
$$\mathbf{A}_{\infty}:=\bigcap_{l\in \mathbb{N}_0} \mathbf{A}_l.$$
For any Borel $B \subset A_0$ with $m(B) >0$ and $l \in \mathbb{N}$, define the $l$th stage density $\delta_l(B, \mathcal{A})$ of $B$ in $\mathcal{A}$ by
\begin{equation*}
\delta_l(B, \mathcal{A})=\frac{m(\mathbf{A}_l \cap B)}{m(B)}.
\end{equation*}
Assume that
\begin{equation} \label{e:tree5}
\Delta_l(\mathcal{A}):=\inf_{B\in \mathcal{A}_l}\delta_{l+1}(B, \mathcal{A}) >0.
\end{equation}

The following lemma, is proved in \cite{Mc} and \cite{U}. We present the proof from \cite{U} here since it is essential to our construction of measures.

\begin{lemma} \label{treelike}
Let $\mathcal{A}$ be defined as above, satisfying conditions \eqref{e:tree1}--\eqref{e:tree5}. Assume that there exist constants $D>0$ and $k>0$ such that for any $z \in A_0$,
\begin{equation*}
m(B(z,r)) \leq Dr^k
\end{equation*}
Denote
\begin{equation*}
\epsilon:=\limsup_{l \to \infty}\frac{\sum_{i=0}^{l}\log(\frac{1}{\Delta_i(\mathcal{A})})}{\log(\frac{1}{d_l(\mathcal{A})})}
\end{equation*}
Then there exists a sequence of measures $\mu^{(l)}$ supported on $\mathbf{A}_l$ such that:
\begin{enumerate}
  \item The sequence $\mu^{(l)}$ has a unique limit measure $\bar{\mu}$, which is supported on $\mathbf{A}_{\infty}$;
  \item $\bar{\mu}(B(z,r)) \leq Cr^{k-\epsilon}$ for any $z\in \mathbf{A}_\infty$, $r>0$, and some $C>0$;
  \item  $\dim_H(\mathbf{A}_{\infty}) \geq k-\epsilon$.
\end{enumerate}

\end{lemma}

\begin{proof}
We define a sequence of Borel measures $\{\mu^{(l)}\}_{l=0}^{\infty}$ on $A_0$ inductively as follows: let $\mu^{(0)}:= m|_{A_0}$, and for each Borel $B \subset A_0$,
\begin{equation}\label{rescale}
\mu^{(l+1)}(B):= \sum_{A\in \mathcal{A}_l}\frac{m(B\cap A \cap \mathbf{A}_{l+1})}{m(A\cap \mathbf{A}_{l+1})}\mu^{(l)}(A).
\end{equation}
Then by induction each $\mu^{(l)}$ is supported on $\mathbf{A}_l$. For any $A\in \mathcal{A}_l$, $\mu^{(l+1)}(A)=\mu^{(l)}(A)$ hence $\mu^{(i)}(A)=\mu^{(l)}(A)$ for any $i \geq l$ by \eqref{e:tree2} and \eqref{e:tree3}. Since $\lim_{l \to 0}d_l(\mathcal{A})=0$, there exists a unique measure $\bar{\mu}$ as the limit measure of $\{\mu^{(l)}\}$ on $A_0$ such that $\bar{\mu}$ is supported on $\mathbf{A}_{\infty}$. Moreover for any $A \in \mathcal{A}_l$ we have $\bar{\mu}(A)=\mu^{(l)}(A)$ and $\bar{\mu}(A) \leq \frac{m(A)}{\prod_{i=0}^{l-1}\Delta_i}$.
Next we prove (2). Consider $y\in \mathbf{A}_{\infty}$, and small $r>0$. Choose $l$ such that $d_{l+1}(\mathcal{A}) \leq r < d_l(\mathcal{A})$. Then $B(y, 2r)$ contains all the sets in $\mathbf{A}_{l+1}$ which meet $B(y, r)$. Thus,
\begin{equation} \label{e:estimate2}
\bar{\mu}(B(y, r)) \leq \bar{\mu}(B(y, 2r)) \leq \frac{m(B(y, 2r))}{\prod_{i=0}^{l}\Delta_i} \leq Cr^{k-\epsilon} [\frac{d_l^{\epsilon}}{\prod_{i=0}^{l}\Delta_i}].
\end{equation}
Whenever $\epsilon$ is greater than $\limsup_{l \to \infty}\frac{\sum_{i=0}^{l}\log(\frac{1}{\Delta_i(\mathcal{A})})}{\log(\frac{1}{d_l(\mathcal{A})})}$, the term in the bracket of \eqref{e:estimate2} goes to zero as $l \to \infty$ and hence bounded above. (3) follows from (2) and Lemma \ref{frostman}.
\end{proof}

The following lemma which is essentially proved in \cite{KW1} Theorem 2.7, relates Lemma \ref{treelike} and Schmidt games.

\begin{lemma} \label{schmidtgame}
Let $X$ be a $k$ dimensional Riemannian manifold and let $S \subset X$ be a $\alpha$-winning set for Schmidt games played on $X$. Then for any open $U\subset X$, one has
\begin{equation*}
\dim_H(S \cap U) = k.
\end{equation*}
\end{lemma}

\begin{proof}
Consider $(\alpha, \beta)$-Schmidt games. We will construct a family $\mathcal{A}$ satisfying \eqref{e:tree1}--\eqref{e:tree5} whose limit set $\mathbf{A}_{\infty}$ is a subset of $S\cap U$. It will be constructed by considering possible moves for Bob at each stage and Alice's corresponding counter moves, which give us different sequences as in \eqref{e:sequence}. We will take the measure $m$ in Lemma \ref{treelike} to be the volume measure $\nu$ on $X$.  Let $\rho_0 >0$ small enough be as in Definition \ref{power} and have the following property: for any $0<\rho < \rho_0$, any ball of radius $\rho$ on $X$ contains at least $N(\beta)$ disjoint balls of radius $\beta\rho$ where $N(\beta) \geq c\beta^{-k}$ for some constant $c>0$.

Bob may begin the game by choosing a ball $B(\omega_1) \subset U$ with $r_1 <\rho_0$. Since $S$ is winning, Alice can choose a ball $B(\omega_1')$ which has nonempty intersection with $S$. Take $A_0:= B(\omega_1')$, hence \eqref{e:tree1} is satisfied.

By the choice of $r_1$, $B(\omega_1')$ contains $N(\beta)$ disjoint balls of radius $\beta r_1'$, say $B(\omega_2^{(1)}),\\ \cdots, B(\omega_2^{(N(\beta))})$, and each of them could be chosen by Bob as $B(\omega_2)$. For each of such choice $B(\omega_2^{(i)})$ of Bob, Alice can pick a ball $B((\omega')_2^{(i)}) \subset B(\omega_2^{(i)})$. Let $\mathcal{A}_1$ be the collection of $N(\beta)$ balls $B((\omega')_2^{(i)})$, i.e., the balls chosen by Alice. Repeating the same for each turn of the game, we obtain $\mathcal{A}_2$, $\mathcal{A}_3$, etc. \eqref{e:tree2} and \eqref{e:tree3} are immediate from the construction above. \eqref{e:tree4} follows from $d_l(\mathcal{A}) = r_1\alpha (\alpha\beta)^l \to 0$ as $l\to \infty$.

Let us verify \eqref{e:tree5}. By Lemma \ref{powerlaw}, one has for any $0<\rho <\rho_0$,
\begin{equation*}
\Delta_l(\mathcal{A}) = \frac{N(\beta)\nu(B(\omega'_{l+2}))}{\nu(B(\omega'_{l+1}))} \geq \frac{c\beta^{-k}c_1(r_1\alpha(\alpha\beta)^{l+1})^k}{c_2(r_1\alpha(\alpha\beta)^{l})^k} \geq \frac{cc_1\alpha^k}{c_2}>0.
\end{equation*}
By Lemma \ref{treelike},
\begin{equation} \label{e:longexpression}
\begin{aligned}
\dim_H(\mathbf{A}_{\infty})
&\geq k-\limsup_{l \to \infty}\frac{\sum_{i=0}^{l}\log(\frac{1}{\Delta_i(\mathcal{A})})}{\log(\frac{1}{d_l(\mathcal{A})}}
\\ &=k-\frac{\log\frac{c_2}{cc_1\alpha^k}}{\log\frac{1}{\alpha\beta}}
\\ & \to k  \ \ \ \text{as\ } \beta \to 0.
\end{aligned}
\end{equation}
\end{proof}

\begin{remark}
The family $\mathcal{A}$ in the construction above depends on the choice of $\beta$, the choice for first move of Bob (in particular $\rho_0$), and choice of all the possible following moves of Alice and Bob.
\end{remark}

Now we combine the constructions appear in proofs of Lemma \ref{treelike} and Lemma \ref{schmidtgame}. Recall the definition of $\Pi(x_0)$ in \eqref{e:set}, and let $x\in \overline{B^W(x_0,\delta/2)}$.
\begin{proposition} \label{familymeasure}
For any $\epsilon >0$, any $x\in \overline{B^W(x_0,\delta/2)}$, there exist a sequence of measures $\{\mu_x^{(l)}\}_{l=0}^{\infty}$ and $\{\mu_x\}$ such that:
\begin{enumerate}
  \item $\mu_x^{(l)}\text{\ is supported on\ } B^u(x, \delta/2), \text{\ and\ }\mu_x \text{\ is supported on\ } B^u(x, \delta/2) \cap E_x(f,y);$
  \item $\mu_x \text{\ is the unique weak limit of\ } \mu_x^{(l)};$
  \item for any $z\in B^u(x, \delta/2)$, $r>0$ small enough,
\begin{equation}\label{e:frostman}
\mu_x (B^u(z, r)) \leq Cr^{u-\epsilon}.
\end{equation}
\end{enumerate}
\end{proposition}

\begin{proof}
For any $\epsilon >0$, we can choose $\beta$ small enough as in proof of Lemma \ref{schmidtgame}, such that the second expression in \eqref{e:longexpression} is greater than $u-\epsilon$ (now $k=u$). For each $x\in\overline{ B^W(x_0,\delta/2)}$, we can construct a family $\mathcal{A}_x=\{\mathcal{A}_{x}^{(l)}\}_{l=0}^{\infty}$ of collections of subsets of $B^u(x, \delta/2)$ with limit set $\mathbf{A}_x^{\infty} \subset B^u(x, \delta/2) \cap E_x(f,y)$ as in the proof of Lemma \ref{schmidtgame}. Follow the proof of lemma \ref{treelike}, we can construct $\mu_x^{(l)}$ supported on $\mathbf{A}_{x}^{(l)}$ and $\mu_x$ supported on $B^u(x, \delta/2) \cap E_x(f,y)$, such that $\mu_x \text{\ is the unique weak limit of\ } \mu_x^{(l)}$. Moreover \eqref{e:frostman} follows from (2) in Lemma \ref{treelike}.
\end{proof}

\begin{remark}
Since the construction of $\mu_x^{(l)}$ and $\mu_x$ in Proposition \ref{familymeasure} relies on the construction of $\mathcal{A}_{x}^{(l)}$, they are not unique.
\end{remark}

\subsection{Measurability of $\{\mu_x\}$}

Recall that the choice for $\mu_x^{(l)}$ and $\mu_x$ are not unique for any $x\in \overline{B^W(x_0, \delta/2)}$. We shall show that there exists a choice for each $\mu_x^{(l)}$ and $\mu_x$ such that $x\mapsto \mu_x^{(l)}$ and $x\mapsto \mu_x$ are measurable with respect to the volume measure $\nu$ on $W$. So we need to specify the choice of the families $\mathcal{A}_x^{(l)}$ for any $l$ and any $x\in \overline{B^W(x_0, \delta/2)}$.

\begin{proposition}\label{measurable1}
There exist a sequence of finite partitions $\mathcal{P}^{(l)}$ of $\overline{B^W(x_0, \delta/2)}$, a family $\mathcal{A}_x=\{\mathcal{A}_{x}^{(l)}\}_{l=0}^{\infty}$ of collections of subsets of $B^u(x, \delta/2)$ and a family of measures $\{\mu_x^{(l)}\}$ such that:
\begin{enumerate}
  \item for each element $P_i^{(l)} \in \mathcal{P}^{(l)}$, the interior of $P_i^{(l)}$ denoted as $\text{Int}(P_i^{(l)})$, is an open and connected set, and $\nu(\partial P_i^{(l)})=0$;
  \item $\mathcal{P}^{(l)} \leq \mathcal{P}^{(l+1)}$;
  \item for each $x\in \overline{B^W(x_0, \delta/2)}$, each $l \in \mathbb{N}_0$, $\mu_x^{(l)}$ is supported on $\mathbf{A}_x^{(l)}$, and obtained as in Proposition \ref{familymeasure}. Moreover, $\bigcup_{x\in \text{Int}(P_i^{(l)})} \mathbf{A}_x^{(l)}$ is a union of finitely many open and connected sets.
  \item For any $l\in \mathbb{N}_0$, $x \mapsto \mu_x^{(l)}$ is continuous on each $\text{Int}(P_i^{(l)})$, hence it is measurable with respect to the volume measure $\nu$ on $\overline{B^W(x_0, \delta/2)}$.
\end{enumerate}
\end{proposition}

\begin{proof}
We prove it by induction on $l$. At first consider $l=0$. For each $x\in \overline{B^W(x_0, \delta/2)}$ let Bob choose a ball $B^u_x(\omega_1)$ satisfying \eqref{e:firstmove} (where $\rho_0$ is as in the proof of Lemma \ref{schmidtgame}) such that $\bigcup_{x\in \overline{B^W(x_0, \delta/2)}}B^u_x(\omega_1)$ is open and connected. Since there is no $I_x^{(k)}$, the $k$th preimages of $\Pi(c)$, for Alice to avoid at the first step, Alice can choose $B^u_x(\omega'_1) \subset B^u_x(\omega_1)$ such that $\bigcup_{x\in \overline{B^W(x_0, \delta/2)}}B^u_x(\omega'_1)$ is open and connected. Set $\mathbf{A}_x^{(0)}=B^u_x(\omega'_1)$, and $\mu_x^{(0)}=\nu|_{B^u_x(\omega'_1)}$. Hence $\mathcal{P}^{(0)}=\{\overline{B^W(x_0, \delta/2)}, \emptyset\}$, and $x \mapsto\mu_x^{(0)}$ is obviously continuous on $B^W(x_0, \delta/2)$.

Suppose the conclusion is true for $0, \cdots, l-1$, and now we prove it for $l$. Since $\mathcal{A}_x^{(l)}$ will be constructed in the $(l+1)$th turn of Schmidt games, let us suppose $jr < l+1 \leq (j+1)r$, i.e., we are in the $j$th step in the proof of Theorem \ref{conformal}. In $j$th step, Alice only needs to avoid some of the $I_k$'s with $R(I_k) \geq \alpha \rho(\alpha\beta)^{(j+2)r-1}$. Since $R(I_k) \leq \frac{c}{\sigma_1^k}$, we have
\begin{equation*}
\alpha \rho(\alpha\beta)^{(j+2)r-1} \leq \frac{c}{\sigma_1^k}.
\end{equation*}
Hence $k \leq \frac{\log \frac{c}{\alpha\rho(\alpha\beta)^{(j+2)r-1}}}{\log \sigma_1}:=K(j)$, that is, $k$ is bounded above when $j$ is fixed. Let $\pi: \Pi(x_0)\rightarrow \overline{B^W(x_0, \delta/2)}$ be the natural map such that $\pi(z)$ is the unique point in $B^u(z,\delta/2) \cap \overline{B^W(x_0, \delta/2)}$. Consider the projection to $\overline{B^W(x_0, \delta/2)}$ under $\pi$ of the boundaries of all these $k$th preimages of $\Pi(c)$, i.e. $\bigcup_{k \leq K(j)} \partial \left(\pi \left(f^{-k}(\Pi(c)) \cap \Pi(x_0)\right)\right)$. Let $\mathcal{Q}^{(l)}$ be the partition of $\overline{B^W(x_0, \delta/2)}$ such that the boundaries of its elements are $\bigcup_{k \leq K(j)} \partial \left(\pi \left(f^{-k}(\Pi(c)) \cap \Pi(x_0)\right)\right)$. Since the number of $k$'s is finite, the partition $\mathcal{Q}^{(l)}$ is finite and $\mathcal{Q}^{(l-1)} \leq \mathcal{Q}^{(l)}$. Set $\mathcal{P}^{(l)}=\mathcal{P}^{(l-1)} \vee \mathcal{Q}^{(l)}$. We proved (1) and (2), and in fact $\mathcal{P}^{(l)}$'s are same for all $l$ in $j$th step, i.e., with $jr < l+1 \leq (j+1)r$.

We prove (3) and (4) now. Consider $P_i^{(l)} \in \mathcal{P}^{(l)}$. Since $\mathcal{P}^{(l-1)} \leq \mathcal{P}^{(l)}$, by the induction $\bigcup_{x\in P_i^{(l)}} \mathbf{A}_x^{(l-1)}$ is a finite union of open and connected sets. So inside each of Alice's balls in $l$th turn, we can let Bob choose the $N(\beta)$ balls $B^u_x(\omega_{l+1})$ at $(l+1)$th turn such that each $\bigcup_{x \in P_i^{(l)}}B^u_x(\omega_{l+1})$ is open and connected. Then by the definition of $\mathcal{P}^{(l)}$, $f^{-k}(\Pi(c)) \bigcap \bigcup_{x\in P_i^{(l)}}B^u_x(\omega_{l+1})$ is a finite union of open and connected sets for each $k\leq K(j)$. Note that in the $(l+1)$th turn, it is enough for Alice to avoid some of $I_k$'s with $k \leq K(j)$, and hence $\bigcup_{x\in P_i^{(l)}}\left(I_x^{(k)}\cap B^u_x(\omega_{l+1})\right)$ is open and connected for each of such $k$. By avoiding the above open and connected sets, Alice can choose balls $B^u_x(\omega'_{l+1})$ such that $\bigcup_{x\in P_i^{(l)}}B^u_x(\omega'_{l+1})$ is open and connected. Alice can do so because of the continuity of the unstable foliation and we can let $\delta$ be small enough in the definition of $\Pi(x_0)$. We proved (3). (4) is immediate from (3) and the fact that $\mu_x^{(l)}$ is obtained by a rescaling of the volume measure in each of the $l$ steps according to \eqref{rescale}.
\end{proof}

\begin{proposition}\label{measurable}
Fix an arbitrary small $\epsilon >0$. For each $x\in \overline{B^W(x_0, \delta/2)}$, there exists a measure $\mu_x$ supported on $B^u(x, \delta/2) \cap E_x(f,y)$, such that
\begin{enumerate}
\item For any $z\in B^u(x, \delta/2)$, $\mu_x (B^u(z, r)) \leq Cr^{u-\epsilon}.$
\item $x\mapsto \mu_x$ is measurable with respect to the volume measure $\nu$ on $\overline{B^W(x_0, \delta/2)}.$
  \end{enumerate}
\end{proposition}

\begin{proof}
As in the proof of Proposition \ref{familymeasure}, we take $\mu_x$ as the unique weak limit of $\mu_x^{(l)}$ in Proposition \ref{measurable1}, and (1) is exactly \eqref{e:frostman}. Since $x \mapsto\mu_x^{(l)}$ is measurable, so is $x\mapsto\mu_x$, we get (2).
\end{proof}

\subsection{Measure $\mu$}
Now we define measure $\mu$ supported on $\Pi(x_0)\cap E(f,y)$.
\begin{definition} \label{measure}
For any Borel set $A \subset M$, define
\begin{equation*}
\mu(A):= \int_{\overline{B^W(x_0, \delta/2)}}\int_{B^u(x, \delta/2)} \chi_A(x,z)d\mu_x(z) d\nu(x).
\end{equation*}
where $\chi_A$ is the characteristic function of $A$, $\nu$ is the volume measure on $\overline{B^W(x_0, \delta/2)}$, and $\mu_x$ is as in Proposition \ref{measurable}. $\mu$ is well defined since we have that $x \mapsto \mu_x$ is measurable with respect to $\nu$.
\end{definition}

Fix $0 <\tau <1$. For any $z \in \Pi(x_0)$, define
\begin{equation*}
B_l(z):=B(z, \tau^l).
\end{equation*}
and
\begin{equation*}
C_l(z):=\bigcup_{u \in B^W (z, \tau^l)}B^u(u, \tau^l).
\end{equation*}

\begin{lemma}
There exists a $l_0 >0$ such that for any $l \geq l_0$, any $z\in \Pi$ we have
\begin{equation*}
 C_{l}(z) \subset B_{l-l_0}(z) \text{\ \ and\ \ } B_{l}(z) \subset C_{l-l_0}(z).
\end{equation*}
\end{lemma}
\begin{proof}
Suppose $w\in C_l(z)$ for some $l$ large enough. Let $u \in W(z) \cap W^u(w)$. Then
\begin{equation*}
d(z, w) \leq d(z, u)+d(u, w) \leq d^W(z, u)+d^u(u, w) \leq 2\tau^{l} \leq \tau^{l-l_0},
\end{equation*}
if we pick some $l_0$ with $2\leq \tau^{-l_0}$. Hence $w\in B_{l-l_0}(z)$. We get the first inclusion.

For the second inclusion, suppose $w\in B_l(z)$ and let $u \in W(z) \cap W^u(w)$. Since the angles between $W^u$ and $W$ are uniformly bounded away from zero, there exists a $C_0>0$ such that
\begin{equation*}
d(z,u) \leq C_0d(z,w), \text{\ and\ \ } d(u,w) \leq C_0d(z,w).
\end{equation*}
It is not hard to prove that for $\delta$ small enough, there exists a constant $C_1>0$ such that $d^u(u,w) \leq C_1d(u,w)$ and $d^W(z,u) \leq C_1d(z,u)$. Thus $d^u(u,w) \leq C_1C_0\tau^l \leq \tau^{l-l_0}$ and $d^W(z,u) \leq C_1C_0\tau^l \leq \tau^{l-l_0}$ for some $l_0$ with $C_1C_0 \leq \tau^{-l_0}$. Therefore, $w\in C_{l-l_0}(z)$, hence $B_{l}(z) \subset C_{l-l_0}(z)$.
\end{proof}

\begin{lemma}\label{frostman2}
For any $r>0$ small enough, any $z\in \Pi(x_0)$, one has
\begin{equation*}
\mu(B(z,r)) \leq D r^{n-\epsilon}
\end{equation*}
for some constant $0<D<\infty$.
\end{lemma}
\begin{proof}
There exists a $l>0$ such that $ \tau^{l+1} \leq r < \tau^l$. Let $h^u$ be the holonomy map along $W^u$ from $W(z)$ to $W(x_0)$. Since $W^u$ is transversally absolutely continuous, then $\exists C_2>1$, for any Borel $A \subset W(z)$, we have
\begin{equation}\label{e:absolute}
C_2^{-1}\nu(A) \leq \nu(h^u(A)) \leq C_2\nu(A).
\end{equation}
Then
\begin{equation*}
\begin{aligned}
\mu(B(z,r)) &< \mu(B(z, \tau^l)) \leq \mu (C_{l-l_0}(z))
\\& = \int_{h^u(B^W(z,  \tau^{l-l_0}))}\int_{B^u((h^u)^{-1}(x),  \tau^{l-l_0})} d\mu_x(w) d\nu(x)
\\& \leq C\int_{h^u(B^W(z,  \tau^{l-l_0}))}( \tau^{l-l_0})^{u-\epsilon} d\nu(x) \text{\ \ \ (by Proposition \ref{measurable}}(1))
\\& \leq CC_2 \nu(B^W(z,  \tau^{l-l_0}))( \tau^{l-l_0})^{u-\epsilon} \text{\ \ \ (by \eqref{e:absolute})}
\\& \leq CC_2C_3( \tau^{l-l_0})^{c+s}( \tau^{l-l_0})^{u-\epsilon}
\\& \leq CC_2C_3( \tau^{l-l_0})^{n-\epsilon}
\\& \leq CC_2C_3( \tau^{-l_0-1})^{n-\epsilon}r^{n-\epsilon} \text{\ \ \ (by\ \ } \tau^{l+1} \leq r)
\\&:=Dr^{n-\epsilon}.
\end{aligned}
\end{equation*}
\end{proof}

\begin{proof}[Proof of Theorem \ref{main2}]
For any nonempty open subset $V \subset M$, we can find some $x_0 \in V$, a local foliation $W$ transversal to $W^u$ near $x_0$, and $\delta>0$ small enough, such that $\Pi(x_0, W, \delta) \subset V$. Theorem \ref{main2} follows immediately from Lemma \ref{frostman} and Lemma \ref{frostman2} and letting $\epsilon \to 0$ in Proposition \ref{measurable}.
\end{proof}

\begin{proof}[Proof of Corollary \ref{countable}]
Let $Y=\{y_t\}_{t=1}^\infty$. Then $E_x(f, Y)=\cap_{t=1}^\infty E_x(f, y_t)$ is also a winning set of Schmidt games played on $W^u(x)$. Let's recall a proof of this fact due to \cite{S}. Let $\alpha$ be as in \eqref{e:winning}. We know that $E_x(f, y_t)$ is $\alpha$-winning for all $t \in \mathbb{N}$, and we want to show that $E_x(f, Y)$ is $(\alpha, \beta)$-winning for any $0<\beta<1$. Here is the strategy for Alice to win. At the first, third, fifth, \ldots turns, Alice uses a $(\alpha, \beta\alpha\beta;E_x(f, y_1))$-winning strategy which forces $\cap_{j=1}^\infty B(\omega_{1+2(j-1)}) \subset E_x(f, y_1)$. At the second, sixth, tenth, \ldots turns, Alice uses an $(\alpha, \beta(\alpha\beta)^3;E_x(f, y_2))$-winning strategy which forces $\cap_{j=1}^\infty B(\omega_{2+2^2(j-1)}) \subset E_x(f, y_2)$. In general, at $k$th turn with $k\equiv 2^{t-1} (\text{mod}2^t)$, Alice uses an $(\alpha, \beta(\alpha\beta)^{2^t-1};E_x(f, y_t))$-winning strategy which forces $\cap_{j=1}^\infty B(\omega_{2^{t-1}+2^t(j-1)}) \subset E_x(f, y_t)$. By this strategy, Alice can enforce that the unique point in the intersection of all balls is in $E_x(f, Y)$.

Now we construct the measure $\mu$ supported on $\Pi(x_0)\cap E(f,Y)$ as in Definition \ref{measure}. We can construct $\mu_x^{(l)}$ and $\mu_x$ as in Proposition \ref{familymeasure} supported on $B^u(x, \delta/2) \cap E_x(f, Y)$ since $E_x(f, Y)$ is a winning set of Schmidt games. So it is enough to specify a choice such that $x\mapsto \mu_x$ is measurable. The idea is same as in the proof of Proposition \ref{measurable1}. The difference is that at $l$th step of the induction(i.e. $(l+1)$th turn of the game) with $l+1\equiv 2^{t-1} (\text{mod}2^t)$, Alice need to avoid some $I_x^{(k)}(y_t)$, the preimages of $\Pi(y_t, c_t)$, where $\Pi(y_t, c_t)$ is the open rectangle neighborhood of $y_t$ as in the proof of Theorem \ref{conformal} for the $(\alpha, \beta(\alpha\beta)^{2^t-1};E_x(f, y_t))$ Schmidt games. If $l+1=2^{t-1}+2^t(j-1)$, then $k$ is bounded above by some number $K(j)$, and there are only finitely many such $k$'s. So the argument in Proposition \ref{measurable1} still works, i.e., at $(l+1)$th turn, there exists a finite partition $\mathcal{P}^{(l)}$, and Alice can choose balls such that $\bigcup_{x\in P_i^{(l)}} \mathbf{A}_x^{(l)}$ is a finite union of open and connected sets. Hence we have that $x\mapsto \mu_x$ is measurable as before.
\end{proof}
\ \
\\[-2mm]
\textbf{Acknowledgement.} The author would like to thank Federico Rodriguez Hertz and Anatole Katok for introduction to the topic and numerous discussions with them. The author would also like to thank Fran\c{c}ois Ledrappier for pointing out some results belonging to the history of the topic. The author also thanks Aaron Brown for helpful discussions.


\begin{thebibliography}{10}

\bibitem{AL}
C. S. Aravinda and E. Leuzinger, \emph{Bounded geodesics in rank-1 locally symmetric spaces}, Ergodic Theory and Dynamical Systems 15.5 (1995): 813-820.


\bibitem{BP}
M.~Brin and Ja.~B.~Pesin. \emph{Partially hyperbolic dynamical systems}, Mathematics of the USSR-Izvestiya 8.1 (1974): 177.

\bibitem{BFK}
R.~Broderick, L.~Fishman, and D.~Y.~Kleinbock, \emph{Schmidt's game, fractals, and orbits of toral endomorphisms}, Preprint, (2010).

\bibitem{BW}
K.~Burns and A.~Wilkinson. \emph{On the ergodicity of partially hyperbolic systems}, Annals of Math., 171:451-489, 2010.

\bibitem{D1}
S. G. Dani, \emph{Divergent trajectories of flows on homogeneous spaces and Diophantine approximation}, J. reine angew. Math 359.55-89 (1985): 102.

\bibitem{D2}
S. G. Dani, \emph{Bounded orbits of flows on homogeneous spaces}, Commentarii Mathematici Helvetici 61.1 (1986): 636-660.

\bibitem{D}
S.~G.~Dani, \emph{On orbits of endomorphisms of tori and the Schmidt game}, Ergodic Theory and Dynamical Systems 8.04 (1988): 523-529.

\bibitem{DS}
S. G. Dani and H. Shah, \emph{Badly approximable numbers and vectors in Cantor-like sets}, Proceedings of the American Mathematical Society 140.8 (2012): 2575-2587.

\bibitem{Dol}
D. Dolgopyat, \emph{Bounded orbits of Anosov flows}, Duke Mathematical Journal 87.1 (1997): 87-114.

\bibitem{F}
K.~Falconer, \emph{Fractal geometry: mathematical foundations and applications}, Wiley. com, 2007.


\bibitem{Fr}
J.~M.~Franks, \emph{Invariant sets of hyperbolic toral automorphisms}, American Journal of Mathematics 99.5 (1977): 1089-1095.

\bibitem{Katok1}
A. Katok, \emph{Nonuniform hyperbolicity and structure of smooth dynamical systems}, Proc. International Congress of Mathematicians, Warszawa. Vol. 2. 1983.

\bibitem{Katok2}
A. Katok and L. Mendoza, \emph{Dynamical systems with nonuniformly hyperbolic behavior}, Supplement to "Introduction to the modern theory of dynamical systems" (A. Katok and B. Hasselblatt) 1995, 659-700.

\bibitem{KM}
D.~Y.~Kleinbock and G.~A.~Margulis, \emph{Bounded orbits of nonquasiunipotent flows on homogeneous spaces}, American Mathematical Society Translations (1996): 141-172.

\bibitem{KW1}
D.~Y.~Kleinbock and B.~Weiss, \emph{Modified Schmidt games and Diophantine approximation with weights}, Advances in Mathematics 223.4 (2010): 1276-1298.

\bibitem{KW2}
\bysame, \emph{Modified Schmidt games and a conjecture of Margulis}, arXiv preprint arXiv:1001.5017 (2010).

\bibitem{Ma}
R.~Ma\~{n}\'{e},  \emph{Orbits of paths under hyperbolic toral automorphisms}, Proceedings of the American Mathematical Society 73.1 (1979): 121-125.

\bibitem{Mc}
C.~McMullen,  \emph{Area and Hausdorff dimension of Julia sets of entire functions}, Transactions of the American Mathematical Society 300.1 (1987): 329-342.

\bibitem{Pr}
F.~Przytycki, \emph{Construction of invariant sets for Anosov diffeomorphisms and hyperbolic attractors}, Studia Mathematica 68.2 (1980): 199-213.

\bibitem{PS}
C.~Pugh, and M.~Shub. \emph{Ergodicity of Anosov actions}, Inventiones mathematicae 15.1 (1972): 1-23.

\bibitem{quassoo}
A.~Quas and T.~Soo. \emph{Ergodic universality of some topological dynamical systems}, arXiv preprint arXiv:1208.3501 (2012).

\bibitem{RRU}
F.~Rodriguez Hertz, M.~A.~Rodriguez Hertz, and R.~Ures, \emph{Accessibility and stable ergodicity for partially hyperbolic diffeomorphisms with 1d-center bundle}, Inventiones mathematicae 172.2 (2008): 353-381.

\bibitem{S}
W.~M.~Schmidt, \emph{On badly approximable numbers and certain games}, Transactions of the American Mathematical Society 123.1 (1966): 178-199.

\bibitem{S2}
W.~M.~Schmidt, \emph{Badly approximable systems of linear forms}, Journal of Number Theory 1.2 (1969): 139-154.

\bibitem{sun1}
P. Sun, \emph{Zero-entropy invariant measures for skew product diffeomorphisms}, Ergodic Theory and Dynamical Systems 30.03 (2010): 923-930.

\bibitem{sun2}
P. Sun, \emph{Measures of intermediate entropies for skew product diffeomorphisms}, Discrete Contin. Dyn. Syst - A, 27 (2010), 1219-1231.

\bibitem{sun3}
P.~Sun, \emph{Density of metric entropies for linear toral automorphisms}, Dynamical Systems 27.2 (2012): 197-204.

\bibitem{U}
M.~Urba\'{n}ski, \emph{The Hausdorff dimension of the set of points with nondense orbit under a hyperbolic dynamical system}, Nonlinearity 4.2 (1991): 385.


\end{thebibliography}
\end{document}